\documentclass[a4paper, 12pt]{amsart} 
\usepackage{amsmath}
\usepackage{amssymb}
\usepackage{amscd}
\usepackage{mathrsfs}
\usepackage{url}
\usepackage[all]{xy}
\usepackage{graphicx}
\usepackage{graphics}
\usepackage{tikz}
\usepackage{ulem}
\usepackage{tikz-cd}
\usetikzlibrary{positioning}
\everymath{\displaystyle}
\theoremstyle{plain} 
\newtheorem{theorem}{\indent\bfseries Theorem}[section]
\newtheorem{lemma}[theorem]{\indent\bfseries Lemma}

\newtheorem{claim}[theorem]{\indent\bfseries Claim}
\newtheorem{conjecture}[theorem]{\indent\bfseries Conjecture}

\makeatletter
\def\section{\@startsection{section}{1}
	\z@{2ex \@plus 1ex \@minus .2ex}
	{1ex \@plus.2ex}
	{\centering\normalfont\bfseries\large}} 
\makeatother

\theoremstyle{definition} 
\newtheorem{definition}[theorem]{\indent\bfseries Definition}
\newtheorem{remark}[theorem]{\indent\bfseries Remark}

\begin{document}

	\pagestyle{plain}
	\thispagestyle{plain}
	
	\title[headline]
	{On the Non-semipositivity of a Nef and Big Line Bundle on Grauert's Example}
	\author[ZHANG YANGYANG]{ZHANG YANGYANG$^{1}$}
	\address{ 
		$^{1}$ Department of Mathematics \\
		Graduate School of Science \\
		Osaka Metropolitan University \\
		3-3-138 Sugimoto \\
		Osaka 558-8585 \\
		Japan 
	}
	\email{sk23162@st.omu.ac.jp}
	
	\renewcommand{\thefootnote}{\fnsymbol{footnote}}

	\renewcommand{\thefootnote}{\arabic{footnote}}

	\begin{abstract}
	We study the relation between semipositivity, nefness, and bigness of line bundles on compact Kähler manifolds. While every nef and big line bundle on a compact Kähler manifold $X$ is positive when ${\rm dim}\,X = 1$, counterexamples are known in higher dimensions. We then focus on the case of dimension two: motivated by a conjecture of Filip and Tosatti, we show that the line bundle on Grauert’s example is nef and big but not semipositive, by explicitly computing its first obstruction class, that is originally introduced by Koike as a generalization of the Ueda class.
	\end{abstract}
	
	\maketitle
	\section{Introduction}	
	Kodaira’s embedding theorem states that if a compact Kähler manifold $X$ admits a positive line bundle $L$, then $X$ can be embedded into a projective space by considering the complete linear system $|L^{\otimes m}|$ for a sufficiently large integer $m$; that is, $L$ is ample. Here, {\it positive} means that a holomorphic line bundle $L$ over $X$ admits a smooth Hermitian metric whose Chern curvature form is a positive (1, 1)-form on $X$; the notion of {\it semipositivity} is defined analogously, requiring the curvature form to be a semipositive. As every ample line bundle admits a smooth Hermitian metric with positive defnite curvature (see Section 2.1), the property of a line bundle being ample is equivalent to its being positive.

	 Since semipositivity is weaker than positivity, it is natural to ask whether there is any relation among semipositivity, nefness and bigness, where nefness and bigness can be regarded as generalizations of ampleness, as reviewed in Section 2.1. Motivated by this observation, we consider the following question(see \cite{Ki}): Does there exist a line bundle that is nef and big, although not semipositive? In the case ${\rm dim}\,X = 1$, it is known that every nef and big line bundle is positive. However, Kim constructed an explicit example (Example 2.14 in \cite{Ki}) of a nef and big line bundle that is not semipositive in the case ${\rm dim}\,X \ge 3$. This naturally shifts attention to the case ${\rm dim}\,X = 2$. In the case, Filip and Tosatti posed the following conjecture in  \cite{FT}:
	\begin{conjecture}
		A line bundle $L$ in the Grauert’s example is nef and big although not semipositive.
	\end{conjecture} For Grauert’s example, we refer to [G, p.~365–366]. It will also be explained in Section 2.2. Conjecture 1.1 claims that this is an example of a line bundle on a projective surface that is nef and big, although not semipositive. Based on this conjecture, Koike showed that a line bundle on a projective surface that is constructed by modifying Grauert’s example is indeed nef and big, although not semipositive, as discussed in [Ko]. The main result of the present paper is as follows:
	\begin{theorem}
		Filip and Tosatti's conjecture is affirmative
	\end{theorem}
 The outline of the proof is as follows. First, we investigate the structure of Grauert’s example and study the line bundle $L$ introduced by Filip and Tosatti on this example. Then we compute its first obstruction class $u_1(Y, X,L)$, where $Y$ is a compact $(\mathrm{K\ddot{a}hler})$ submanifold of $X$. In [Ko], the defnition of $u_1(Y, X,L)$ was originally introduced as a generalization of the Ueda class in \cite{U}. Finally, we deduce that $L$ is not semipositive by applying Theorem 1.4 of [Ko].

The organazation of this paper is as follows. In Section 2.1, we recall some necessary definitions. In Section 2.2, we describe in detail the construction of Grauert's example. In Section 2.3, we introduce the first obstruction class in $u_1(Y,X,L)$ associated  with the tripe $(Y,X,L)$. In Section 3, we give the proof the main theorem.
	\vskip3mm
	{\bf Acknowledgment:} 
	The author would like to express his sincere gratitude to Professor Takayuki Koike for his invaluable guidance and continuous support throughout the
	preparation of this paper. He is also grateful to Professor Valentino Tosatti for his insightful comments on the paper. The author further thanks Jinichiro Tanaka for helpful comments and discussions. This work was supported by the JST spring program.

	\section{Preliminaries}
	\subsection{Terminology and Notation}
	In this subsection, we introduce some defnitions and notation in complex geometry. For the details, refer to [Dem2, Chapter V], [GH, Chapter 0] for example. In the present paper, we use the following notation for a holomorphic line bundle $L$ on a complex manifold $X$, submanifold $Y\subset X$, and an open subset $U\subset X$:
\begin{itemize}
		\item $\mathbb{C}$ : the field of complex numbers,
		\item $\mathbb{Z}$ : the field of real integers,
		\item ${\rm deg}\,L$ : the degree of the line bundle $L$,
		\item $\mathrm{dim}\,X$ : the complex dimension of $X$,
		\item $K_X$: the canonical line bundle of $X$,
		\item $N_{Y/X}$: the normal bundle of $Y$ in $X$,			 
		\item $L^{m}$: For an integer $m \in \mathbb{Z}$, we denote by $L^{\otimes m}$ the $m$-th tensor power of $L$. Note that $L^{-1}$ denotes the dual line bundle of $L$.
		\item $H^i(X, \mathcal{F})$ : the $i$-th sheaf cohomology group of $\mathcal{F}$ on $X$, where $\mathcal{F}$ is one of following:
		\item $\mathcal{O}_X$ defined by $\mathcal{O}_X(U)$:=\{all holomorphic functions on $U$\},
		\item $\mathcal{O}^*_X$ defined by $\mathcal{O}^*_X(U)$:=\{all nowhere vanishing holomorphic functions on $U$\},
		\item $\mathcal{O}_X(L)$ defined by $\mathcal{O}_X(L)(U)$ :=\{all holomorphic sections of $L$ on $U$\}.
	\end{itemize}		
	
\begin{definition}[{\it Positive/Semipositive}]
	A line bundle $L$ on $X$ is said to be {\it positive} (resp. {\it semipositive}) 
	if there exists a Hermitian metric $h$ on $L$ such that the Chern curvature form 
	$\Theta_{L,h}$ is positive (resp. semipositive) everywhere.
\end{definition}
	
	\begin{definition}[{\it Nef}]
		Let $X$ be a projective variety and let $L$ be a holomorphic line bundle on $X$. We say that $L$ is {\it nef} if for any algebraic curve $C \subset X$, $(L.C)=\int_{C} c_{1}(L) \ge 0$.
	\end{definition}
	
	\begin{definition}[{\it Big}]
		Let $L$ be a line bundle on a projective variety $X$. The line bundle $L$ is called {\it big} if $\limsup_{m \to \infty} \frac{\dim H^0(X, L^{m})}{m^n} > 0$, where $n={\rm dim}\, X$ is the dimension of the variety $X$.
	\end{definition}
	
	\begin{definition}[{\it Topologically trivial line bundle}]
		Let $X$ be a complex manifold. A line bundle $L$ on $X$ is called {\it topologically trivial} if it is $C^{\infty}$ isomorphic to the trivial bundle $X \times \mathbb{C} \to X$. Equivalently, the first Chern class of $L$ vanishes, i.e.\ $c_{1}(L) = 0\in H^2(X,\mathbb{Z})$.
	\end{definition}
	
	As is well known, if $X$ is a projective variety and $L$ is a holomorphic line bundle over $X$, the following implications hold:
	
\begin{tikzcd}[row sep=4.5em, column sep=4em,
	every node/.style={draw=black, thick, inner sep=4pt}]
	& \text{$L$: ample}
	\arrow[r, Leftrightarrow, "(1)" above]
	& \text{$L$: positive}
	\arrow[r, Rightarrow, "(2)" above]
	\arrow[d, Rightarrow, "(4)" left]
	& \text{$L$: semipositive}
	\arrow[r, Rightarrow, "\text{(3)}" above]
	& \text{$L$: nef} \\
	& & \text{$L$: big}
	& \text{$L$: topologically trivial}
	\arrow[u, Rightarrow, "(5)" right] &
\end{tikzcd}
	\begin{center}
	\textbf{Figure 1}
	\noindent
	\begin{enumerate}
		\item The equivalence between ampleness and positivity follows from Kodaira’s embedding theorem and  \cite[Example~3.14 ]{D1}.
		\item Positivity implies semipositivity by definition.
		\item Semipositivity implies nefness. See \cite[Proposition~6.10]{D1}.
		\item A positive line bundle is big. This is explained in \cite[Corollary~6.19]{D1}.
		\item The implication from topologically trivial to semipositive follows from Kashiwara’s theorem. See Theorem~2.5 below.
	\end{enumerate}
	
	\begin{theorem}[{\cite[Proposition~1, p.~584]{U}}]
	Let $X$ be a compact Kähler manifold, and let $L$ be a holomorphic line bundle over $X$. 
	If $L$ is topologically trivial, then it admits a unitary flat structure. That is, there exists a Hermitian metric $h$ on $L$ such that $\Theta_{L,h} = 0$ at every point of $X$.
\end{theorem}
\end{center}
	\subsection{Grauert's example}
In this subsection, we introduce the construction in Grauert’s example following \cite[p. 365--366]{G}, which will be used in later calculations, together with some related results.
	Let $R$ be a compact Riemann surface of genus $g \geq 2$, and let $F$ be a holomorphic line bundle over $R$ with ${\rm deg}\,F=1$. Assume that $H^1(R, \mathcal{O}_R(F)) \neq 0$.
	\begin{lemma}
		There exists a line bundle $F$ of degree $1$ such that $H^1(R, \mathcal{O}_R(F)) \ne 0$.
	\end{lemma}
	
	\begin{proof}
		Take a point $p \in R$ and set $F = [\{p\}]$, the line bundle associated with the divisor $\{p\}$.  
		By definition,
		\[
		\Gamma(R, \mathcal{O}_R([\{p\}])) = 
		\{ f : R \to \mathbb{C} \text{ meromorphic } \mid \mathrm{div}(f) + \{p\} \ge 0 \}
		\supseteq \mathbb{C},
		\]
		so $\dim H^0(R, \mathcal{O}_R([\{p\}])) \ge 1$.  
		By the Riemann–Roch theorem,
		\[
		\dim H^0(R, \mathcal{O}_R([\{p\}])) 
		- \dim H^1(R, \mathcal{O}_R([\{p\}])) 
		= 1 - g + \deg(F) = 2 - g.
		\]
		Since $g \ge 2$, hence
		\[
		\dim H^1(R, \mathcal{O}_R([\{p\}])) 
		\ge \dim H^0(R, \mathcal{O}_R([\{p\}])).
		\]
		As $\dim H^0(R, \mathcal{O}_R([\{p\}])) \ge 1$, it follows that $H^1(R, \mathcal{O}_R(F)) \ne 0$.
	\end{proof}

Take a sufficiently fine finite open covering $\mathcal{U} = \{ U_j \}_{j\in\Lambda}$ of $R$, and choose a Čech cohomology class
\[
		[\xi] = [\{ \left(U_{jk},\xi_{jk} m_j\right) \}]\in \check{H}^1(\mathcal{U},\mathcal{O}_R(F))
	\]
	such that $[\xi] \neq 0$. Here, each $\xi_{jk} : U_{jk} \to \mathbb{C}$ is a holomorphic function, and $\{ m_j \}$ is a system of local frames of $F$ satisfying
\[
		m_j = a_{jk}^{-1} m_k \quad \text{on } U_j \cap U_k,
\]
	where $a_{jk}$ is the transition function of $F$. Let $\Psi_j: \pi^{-1}(U_j)\to U_{j}\times\mathbb{C}$ be a local trivializations of the total space of the line bundle $F$ over each open set $U_{j}$, where $\pi$ is the projection map. For any $j, k$ and $(x_j, \eta_j) \in U_j \times \mathbb{C}$, $(x_k, \eta_k) \in U_k \times \mathbb{C}$, the transition $a_{jk}$ are given by
	\[
		\Psi_j^{-1}(x_j, \eta_j) = \Psi_k^{-1}(x_k, \eta_k)
		\quad \text{if and only if} \quad
		\begin{cases}
			x_j = x_k \\
			\eta_j = a_{jk}(x) \eta_k
		\end{cases},
\]
	We now construct an affine line bundle $p: A \to R$ associated to the pair $(F, \xi)$. For each $j$, define $V_j:= U_j \times \mathbb{C}$, and define
	$A = A_{F, \xi} := \left( \bigsqcup_{j=1}^r V_j \right) \Big/ \sim$, where for $(x_j, \eta_j) \in U_j \times \mathbb{C}$ and $(x_k, \eta_k) \in U_k \times \mathbb{C}$, we declare
	\[
		(x_j, \eta_j) \sim (x_k, \eta_k) \quad \Longleftrightarrow \quad
		\begin{cases}
			x_j = x_k \\
			\eta_j = a_{jk}(x_j) \eta_k + \xi_{jk}(x_j)
		\end{cases}.
	\]
This defines a holomorphic affine line bundle $A$ over $R$ with transition data determined by the pair $(F, \xi)$. To compactify $A$, we add the “section at infinity” $Y \subset X$, thereby obtaining a compact complex surface $X$. Let $[Y]$ be the line bundle associated to the divisor $Y \subset X$, and define the line bundle $L := p^*F \otimes [Y]$, where $p: X \to R$ is the natural projection. Let $W_j = p^{-1}(U_j) \cap Y = \{ (x_j, \eta_j) \mid \eta_j = +\infty \} = \{ (x_j, \theta_j) \mid \theta_j = 0 \},$ where \( \theta_j = \frac{1}{\eta_j} \). The domain of \( \theta_j \) is (\( V_j \cup Y) \setminus \{ \eta_j = 0 \} \).

	\begin{definition}
		We define a map 
		\[
		\Phi: N_{Y/X} \longrightarrow (p|_{Y})^{*}F^{-1}
		\]
		as follows. For each $W_{j}$ and any point $x \in W_{j}$, we define
		\[
		\Phi_{x}: (N_{Y/X})_{x} \longrightarrow \bigl((p|_{Y})^{*}F^{-1}\bigr)_{x} 
		\text{by}\,\, \Phi_{x}\!\left(f_{j}(x)\cdot\frac{\partial}{\partial\theta_{j}}(x)\right)
		= f_{j}(x)\cdot (p|_{Y})^{*}n_{j}.
		\]
	\end{definition}
	
	\begin{claim}
		The pointwise definition of $\Phi_{x}$ is independent of the choice of local coordinates and defines a global morphism.
	\end{claim}
	
	\begin{proof}
		Since $N_{Y/X} = [Y]|_{Y}$, we have the following identity on $W_{j}\cap W_{k}$
		\[
		\left.\frac{\partial\theta_{j}}{\partial\theta_{k}}
		= \frac{a_{jk}(x_{j})}{\bigl(a_{jk}(x_{j})+\xi_{jk}\theta_{k}\bigr)^{2}}
		\right|_{\theta_{k}=0}
		= \frac{1}{a_{jk}(x_{j})}.
		\]
		Hence,
		\[
		\frac{\partial}{\partial\theta_{j}}
		= \frac{\partial\theta_{k}}{\partial\theta_{j}}\cdot\frac{\partial}{\partial\theta_{k}}
		= a_{jk}(x_{j})\cdot\frac{\partial}{\partial\theta_{k}}.
		\]
	Let $\{(p|_{Y})^{*}n_{j}\}$ be a system of local frames of $(p|_{Y})^{*}F^{-1}$ with the relation
		\[
		(p|_{Y})^{*}n_{j} = a_{jk}\cdot (p|_{Y})^{*}n_{k}.
		\]
		On each $W_{j}$, any section $s$ of $N_{Y/X}$ can be written as
		\[
		s = f_{j}\cdot \frac{\partial}{\partial\theta_{j}},
		\]
		where $f_j$ is a    holomorphic function on $W_j\cap W_k$, it follows that
		\[
		s = f_{k}\cdot \frac{\partial}{\partial\theta_{k}}
		= f_{j}\cdot \frac{\partial}{\partial\theta_{j}}
		= f_{j}a_{jk}\cdot \frac{\partial}{\partial\theta_{k}}.
	    \]
		Now we have
		\[
		\Phi_{x}\left(f_{j}(x)\cdot\frac{\partial}{\partial\theta_{j}}(x)\right)
		= f_{j}(x)\cdot (p|_{Y})^{*}n_{j}(x),
		\]
		and
		\[
		\Phi_{x}\!\left(f_{k}(x)\cdot\frac{\partial}{\partial\theta_{k}}(x)\right)
		= f_{k}(x)\cdot (p|_{Y})^{*}n_{k}(x)
		= f_{j}(x)a_{jk}\cdot (p|_{Y})^{*}n_{k}(x)
		= f_{j}(x)\cdot (p|_{Y})^{*}n_{j}(x).
		\]
		
		Thus the definition is consistent across overlaps, and $\Phi$ is well defined globally.
	\end{proof}
	\begin{claim}[{\cite[example(d), p. 365--366]{G}}]
		$N_{Y/X} \cong(p|_{Y})^{*}F^{-1}$
	\end{claim}
	\begin{proof}
The map $\Phi$ provides an isomorphism between them.
		\end{proof}
	\begin{remark}
	By \cite[Satz~8, p.~353]{G}, if $M$ is a compact Kähler manifold, and \( A \subset M \) is a compact Kähler submanifold whose the normal bundle \( N_{A/M} \) negative, then \( A \) is exceptional. Since, by Claim~2.9 we already know that ${\rm deg}\,N_{Y/X} = {\rm deg}\,F^{-1} = -1$, it follows that $Y$ is exceptional; that is, $X$ admits a proper holomorphic map onto a complex space $Z$. Furthermore, according to \cite[Example~(d),  p.~365--366]{G}, the resulting space $Z$ is not a projective algebraic variety.
\end{remark}

\begin{remark}
In the original Grauert example, the compact complex surface $X$ is obtained as a compactification of an affine line bundle $A$ over a compact Riemann surface $R$, associated with a holomorphic line bundle $F$ of degree 1 and a non-trivial class in $H^{1}(R,\mathcal{O}_R(F))$, by adding an infinity section $Y$. In Koike’s modified Grauert example, one instead considers an affine line bundle associated with the canonical bundle $K_R$ and a non-trivial class in $H^{1}(R,K_R)$, whose natural compactification yields a compact complex surface $X$, which is a ruled surface in the birational sense, with the infinity section $Y$. Koike’s approach is based on an explicit construction: using the \v{C}ech--Dolbeault correspondence, he constructs a specific non-trivial class $\xi \in H^{1}(R,K_R)$ and shows that the first obstruction class $u_1(Y,X,L)$ coincides with this class.
In contrast, in the original Grauert setting, we start from the fact that $H^{1}(R,\mathcal{O}_R(F))$ is non-trivial. Rather than constructing a cohomology class, we take an arbitrary non-zero element $\xi \in H^{1}(R,\mathcal{O}_R(F))$ to define the associated affine line bundle, and show that the non-vanishing of the first obstruction class $u_1(Y,X,L)$ is equivalent, via $N_{Y/X}^{-1}\cong (p|_{Y})^{*}F$, to the non-triviality of $\xi$.
\end{remark}

	\subsection{The first obstruction class}
As introduced in \cite{Ko}, consider a complex manifold $X$ of dimension $d+r$ and a compact K\"ahler submanifold $Y \subset X$ of dimension $d$ ($d, r \ge 1$). Suppose that $L$ is a holomorphic line bundle on $X$ such that the restriction $L|_Y$ is unitary flat. In \cite{Ko}, $u_{1}(Y,X,L)$ was originally defined as a generalization of the Ueda class in \cite{U}, so that when $c_1(N_{Y/X})=0$ and $L=[Y]$, Ueda's first obstruction class can be identified with $u_1(Y,X,L)$. To make the paper self-contained, we describe it in this subsection. Let $X$ be a complex manifold of dimension $n\in\mathbb{Z}_{>0}$, and let  $Y \subset X$ be a compact Kähler submanifold of codimension $r\in\mathbb{Z}_{>0}$. Let $L$ be a holomorphic line bundle on $X$ such that $L|_Y$ is topologically trivial. Take a sufficiently fine open covering $\{V_j\}$ of a neighborhood of $Y$ in $X$, and a sufficiently fine small Stein open sbuset that $W_j$ such that $V_j\cap Y=W_j$. Here $x_j$ are coordinates on $W_j$, and $\theta_j=(\theta_j^1,\theta_j^2,....,\theta_j^r) \in \mathbb{C}^r$ are defining functions of $Y$. Consider the case of $r=1$($\theta_{j}=\theta_{j}^{1})$. Let $\{e_j\}$ be local frames of $L$ over $\{V_j\}$ such that, on the intersection $W_{jk} = W_j \cap W_k$, the frame transforms as
	\[
		t_{jk}^{-1} e_k = e_j \quad \text{with} \quad t_{jk} \in U(1).
	\]
	Then the transition function between local trivializations of $L$ satisfies (on $V_{jk}$)
	\[
		\frac{t_{jk}^{-1} e_k}{e_j} = 1 +f_{kj,1}(x_j)\cdot\theta_{j}+ O(|\theta_{j}|^2).
	\]
 The first obstruction class is the class
	\[
		u_1(Y, X, L) :=\left[\left\{\left(W_{jk}, f_{kj,1}(x_j)\cdot d\theta_j\right)\right\} \right]  \in \check{H}^1(\{W_j\}, \mathcal{O}_{Y}(N^{-1}_{Y/X})),
	\]
	where $f_{kj,1}(x_j)$ is the coefficients in the expansion of ${e_k}/{e_j}$, this defines an element of the Čech cohomology group $\check{H}^1(\{W_j\}, \mathcal{O}_{Y}(N^{-1}_{Y/X}))$, which is called the first obstruction class $u_1(Y, X, L)$.
	
	\begin{theorem} [{\cite[Lemma 3.1]{Ko}}]
		The collection $\{(W_{jk},f_{jk,1}(x_j)\cdot d\theta_{j})\}$ satisfies the 1-cocycle condtion. Moreover, the definition of $u_1(Y,X,L)$ does not depend on the choice of the local coordinates $\theta_{j}$, the local frames $\{e_j\}$, or the open covering $\{V_j\}$ of $Y$. 
	\end{theorem}
	\begin{proof}
		This have already checked by Koike in \cite{Ko}. 
		For the specific example treated in this paper, the fact that 
		\(\{ (W_{jk}, f_{kj,1}(x_j)\cdot d\theta_j) \}\) 
		satisfies the $1$-cocycle condition is also discussed in Claim~3.3 of the next section.
	\end{proof}
	Next, we want to introduce the important theorem in \cite{Ko} which will be used in our main theorem.     
	\begin{theorem}[{\cite[Theorem 1.4]{Ko}}]\label{Theorem 2.12}
		Let $X$ be a complex manifold, and let $Y \subset X$ be a compact Kähler submanifold. Let $L$ be a line bundle on $X$ such that the restriction $L|_Y$ is topologically trivial. Suppose that $u_1(Y, X, L) \neq 0$. Then $L$ is not semipositive.
	\end{theorem}

	\section{Proof of main theorem}
	In this section, we first recall that well-known fact that the line bundle $L$ is nef and big. We then turn to the main result of this paper, which states that $L$ is not semipositive. This will be shown by establishing that the first obstruction class $u_{1}(Y,X,L)$ is nonzero, and by applying Theorem 2.13. All relevant notions used here are as defined in Section 2.1 and Section 2.2.
	\begin{claim}
		The line bundle $L$ is nef and big.
	\end{claim}
	\begin{proof}
		Any effective curve $\Gamma$ in $X$ falls into one of the following two cases:
		
		\textbf{Case 1($\Gamma = Y$):}
		$(L .\Gamma) = (L .Y) = (p^*F \otimes [Y].Y) 
		= {\rm deg}\,p^*F|_Y + {\rm deg}\,[Y]|_Y
		= {\rm deg}\,p^*F|_Y + {\rm deg}\,N_{Y/X} 
		= {\rm deg}\,F - {\rm deg}\,F = 0$ by Claim 2.9.
		
		\textbf{Case 2($\Gamma \ne Y$ and $\Gamma \cap Y \ne \emptyset$):}
		We have $(L .\Gamma) = (p^*F \otimes [Y].\Gamma) = (p^*F .\Gamma) + ([Y] .\Gamma) .$ If $\Gamma\cap Y\ne 0$, then ${\rm deg}\,F=1$ and {\cite[Proposition, p. 148]{GH}} imply that $F$ is positive. By the Figure 1 in Section 2.1, $p^*F$ is semipositive and hence nef, so $(p^*F .\Gamma) \ge 0$. Moreover, since $([Y].\Gamma)>0$, we have $(L.\Gamma)>0$. If $\Gamma\cap Y=0$, then $([Y].\Gamma)=0$ and $p^*F$ being nef gives $(p^*F.\Gamma)\ge 0$, hence $(L.\Gamma) = (p^*F \otimes [Y].\Gamma) = (p^*F.\Gamma) + ([Y].\Gamma) \ge 0.$ 
		
		Combining the results of the two cases above, we conclude that $L$ is nef. To check that $L$ is big, it suffices to show that the self-intersection number of $L$ is positive by {\cite[Corollary~8.4]{D1}}. 
		We compute
		\[
		(L^2) = (L \cdot L) = (p^*F \cdot p^*F) + \deg F = 0 + 1 = 1.
		\]
		This proves that $L$ is big.
	\end{proof}
	
	\begin{theorem}\label{Theorem 3.2}
		The line bundle $L$ is not semipositive.
	\end{theorem}
	
	\begin{proof} 
		Using the transition relation between \( \eta_j \) and \( \eta_k \), we have
		\[
			\theta_j = \frac{1}{\eta_j} = \frac{1}{a_{jk} \eta_k + \xi_{jk}} = \frac{\theta_k}{a_{jk} + \xi_{jk} \theta_k}
		\]
		on \( W_j \cap W_k \). Applying the Taylor expansion of \( \frac{1}{1 + x} \), we obtain
		\[
			\frac{\theta_k}{a_{jk} + \xi_{jk} \theta_k} = \frac{\theta_k}{a_{jk}} \left( 1 - \frac{\xi_{jk}}{a_{jk}} \theta_k + o(\theta_k) \right) = \frac{\theta_k}{a_{jk}} + o(\theta_k).
		\]
		Let $ \{ v_k \} $ be a system of local frames for the line bundle $[Y]$ such that \( v_j \theta_j = v_k \theta_k \) on \( W_j \cap W_k \), which patch to define the canonical section of \([Y] \).  
		It follows that \( m_j = a_{jk}^{-1} m_k \), and
		\[
			v_k = \frac{1}{a_{jk} + \xi_{jk} \theta_k} \cdot v_j.
		\]
		Now, letting $ e_k := p^*m_k \otimes v_k $, we use $\{e_k\}$ as a local frame of \( L \) on \( V_k \) in what follows. Then we compute the transition function
		\[
			\frac{e_k}{e_j} = \frac{p^* m_k \otimes v_k}{p^* m_j \otimes v_j}
			= \frac{a_{jk} \cdot v_k}{v_j}
			= a_{jk} \cdot \frac{1}{a_{jk} + \xi_{jk} \theta_k}.
		\]
		By Taylor expansion, this becomes
		\begin{equation}
		\frac{e_k}{e_j} 
		= 1 - \frac{\xi_{jk}}{a_{jk}}\cdot\theta_k + o(\theta_k).
		\end{equation}
	According to the definition of the first obstruction class in \cite{Ko}, we focus on the coefficient of $\theta_j$. To relate \( \theta_j \) and \( \theta_k \), recall again:
		\begin{equation}
			\theta_j = \frac{\theta_k}{a_{jk} + \xi_{jk} \theta_k}.
		\end{equation}
	By computation, we have
		\[
			\theta_{k}=a_{jk}\cdot\frac{\theta_j}{1-\xi_{jk}\theta_{j}} 
		\]
		By Talor expansion which implies
		\[
			\theta_k = a_{jk} \theta_j \left( 1 - \xi_{jk} \theta_j + o(\theta_j) \right) = a_{jk} \theta_j + o(\theta_j).
		\]
		Substituting this expansion into the transition function computed earlier in Equation (1), we obtain	\[
			\frac{e_k}{e_j} 
			= 1 - \frac{\xi_{jk}}{a_{jk}} \theta_k + o(\theta_k) 
			= 1 - \frac{\xi_{jk}}{a_{jk}} \cdot a_{jk} \theta_j + o(\theta_j) 
			= 1 - \xi_{jk} \theta_j + o(\theta_j).
		\]
		According to our strategy, we aim to apply Theorem 2.13, which was explicitly introduced in Section 2.3.
		In order to prove that $L$ is not semipositive, it suffices to show that the first obstruction class is nontrivial.
		Since the following two claims establish this fact, the proof is complete.
		\end{proof}

	\begin{claim}[{\cite[Lemma 3.1]{Ko}}]
		$\left[ \{\left(W_{jk}, -\xi_{jk}\, d\theta_j \right)\} \right] \in \check{H}^1\left( \{ W_{jk} \}, \mathcal{O}_Y(N_{Y/X}^{-1}) \right).$
	\end{claim}	
	
	\begin{proof}
		We refer the reader to \cite{Ko} for the full details of the proof. However, to make the paper self-contained, we briefly outline the key idea below. By examining Equation (1), we observe that $\frac{e_k}{e_j} \cdot \frac{e_j}{e_l} \cdot \frac{e_l}{e_k} = 1$, then we obtain the identity
		\begin{equation}
			1 - \xi_{jk} \theta_j - \xi_{jl} \theta_l - \xi_{lk} \theta_k = 1.
		\end{equation}
		From Equation (1), we know that
		\[
			\theta_j = \frac{1}{a_{jk}} \cdot \frac{\theta_k}{1 + \frac{\xi_{jk}}{a_{jk}} \theta_k}
			= \frac{\theta_k}{a_{jk}} \cdot \left( 1 - \frac{\xi_{jk}}{a_{jk}} \theta_k + o(\theta_k) \right)
			= \frac{\theta_k}{a_{jk}} + o(\theta_k).
		\]
		Similarly, using the transition functions, we obtain $\theta_l = \frac{\theta_k}{a_{lk}} + o(\theta_k).$ Substituting these into Equation (3), we obtain
		\[
			-\frac{\xi_{jk}}{a_{jk}} - \frac{\xi_{kl}}{a_{kl}} \cdot \frac{1}{a_{lk}} - \frac{\xi_{lj}}{a_{lj}} \cdot \frac{1}{\theta_k} = 0.
		\]
		Our goal is to show that the 1-cochain $-\frac{\xi_{jk}}{a_{jk}} \cdot d\theta_k$ satisfies the Čech 1-cocycle condition, namely
		\[
			-\frac{\xi_{jk}}{a_{jk}} \cdot d\theta_k
			- \frac{\xi_{kl}}{a_{kl}} \cdot d\theta_l
			- \frac{\xi_{lj}}{a_{lj}} \cdot d\theta_j = 0.
		\]
		Considering the local frame of \( N_{Y/X}^{-1} \), we conclude that the cocycle condition is satisfied, as desired.
	\end{proof}

	\begin{claim}\label{Claim 3.4}
		$u_{1}(Y,X,L)\ne 0$.
	\end{claim}
	
	\begin{proof}
		Assume, for contradiction, that $\bigl[ \{(W_{jk}, -\xi_{jk}\, d\theta_j)\} \bigr] = 0\in \check{H}^1\left( \{ W_{jk} \}, \mathcal{O}_Y(N_{Y/X}^{-1}) \right).$ Then for all $j$, there exist holomorphic functions $\varphi_j : W_j \to \mathbb{C}$ such that for all $j,k$,
		\[
		\varphi_k\, d\theta_k - \varphi_j\, d\theta_j = -\xi_{jk}\, d\theta_j 
		\quad \text{on } W_j \cap W_k.
		\]
		Equivalently, we can write
	$\varphi_k \cdot \frac{d\theta_k}{d\theta_j} - \varphi_j = -\xi_{jk},$
	and by local frame of conormal bundle $N_{Y/X}^{-1}$, we get $\varphi_k \cdot a_{jk} - \varphi_j = -\xi_{jk},$ and using the local frame $m_j$ of $F$, we have
		\[
		\varphi_k m_k - \varphi_j m_j = -\xi_{jk} m_j
		\quad \text{on } W_j \cap W_k.
		\]
		Setting $g_j := \varphi_j \circ p^{-1}|_{W_j}$, we obtain
		\[
		\delta\{ (U_j, g_j m_j) \} = \{ (U_{jk}, -\xi_{jk} m_j) \} \text{ on } U_j\cap U_k.
		\] 
		Hence
		\[
		\bigl[ \{ (U_{jk}, -\xi_{jk} m_j) \} \bigr] = 0 
		\quad \text{in } H^1(\{ U_{jk} \}, \mathcal{O}_R(F)),
		\]
		contradicting the assumption that 
		\([\{ (U_{jk}, \xi_{jk} m_j) \}] \ne 0\).
		Therefore $u_1(Y,X,L)\ne 0$.
	\end{proof}
	
		\begin{proof}[Proof of Theorem~\ref{Theorem 3.2}]
			Theorem~\ref{Theorem 2.12} tells us that if the first obstruction class $u_{1}(Y,X,L)$ is nontrivial, then $L$ is not semipositive. Hence, Theorem~\ref{Theorem 3.2} follows from Claim~\ref{Claim 3.4} and Theorem~\ref{Theorem 2.12}.
		\end{proof}

	\section{Discussion and an Open Question}
	
	We conclude by recording a related question that was suggested to us by Valentino Tosatti.
	
	\noindent
	\textbf{Question.}
	In the (original or modified) Grauert example, does the cohomology class $c_1(L)$ contain a closed positive $(1,1)$-current with minimal singularities whose local potentials are bounded?

	The motivation for this question is the following. At present, we do not know any explicit example of a nef and big $(1,1)$-class on a compact K\"ahler manifold which is not semipositive, but nevertheless contains a closed positive current with bounded local potentials. We believe that such examples should exist, and likely in abundance,
	although no concrete construction appears to be available in the literature so far.


\begin{thebibliography}{99}
	
	\bibitem[D1]{D1} 
	\textsc{J.-P.~Demailly}, 
	“Complex Analytic and Differential Geometry", \url{https://www-fourier.ujf-grenoble.fr/~demailly/manuscripts/agbook.pdf}.
	
	\bibitem[D2]{D2} 
	\textsc{J.-P.~Demailly}, 
	Analytic Methods in Algebraic Geometry, International Press, Somerville, MA; Higher Education Press, Beijing, 2012.
	
	\bibitem[FT]{FT} 
	\textsc{S.~Filip and V.~Tosatti}, 
	Smooth semipositive representatives of big classes, Ann. Inst. Fourier \textbf{68} (2018), no.~5, 1981--1999.
	
	\bibitem[G]{G} 
	\textsc{H.~Grauert}, 
	\"{U}ber Modifikationen und exzeptionelle analytische Mengen, Math. Ann. \textbf{146} (1962), 331--368.
	
	\bibitem[GH]{GH} 
	\textsc{P.~A.~Griffiths and J.~Harris}, 
	Principles of Algebraic Geometry, 
	Wiley classics Library. John Wiley and Sons Inc. (1994).

	\bibitem[Ha]{Ha} 
	\textsc{R.~Hartshorne}, 
	Algebraic Geometry, Graduate Texts in Mathematics, vol.~52, 
	Springer-Verlag, New York, 1977.
	
	\bibitem[Hu]{Hu} 
	\textsc{D.~Huybrechts}, 
	Complex Geometry: An Introduction, Springer-Verlag, Berlin, 2005.
	
	\bibitem[Ki]{Ki} 
	\textsc{D.~Kim}, 
	Extension of Pluriadjoint Sections from a Log-Canonical Center, Ph.D. Thesis, Princeton University, 2007.
	
	\bibitem[Ko]{Ko} 
	\textsc{T.~Koike}, 
	Linearization of transition functions of a semipositive line bundle along a certain submanifold, Ann. Inst. Fourier \textbf{71} (2021), no.~5, 2237--2271.
	
	
	\bibitem[L1]{L1} 
	\textsc{R.~Lazarsfeld}, 
	Positivity in Algebraic Geometry~I, Ergebnisse der Mathematik und ihrer Grenzgebiete., \textbf{48}, Springer-Verlag(2004).
	
	\bibitem[L2]{L2} 
	\textsc{R.~Lazarsfeld}, 
	Positivity in Algebraic Geometry~II, Ergebnisse der Mathematik und ihrer Grenzgebiete., \textbf{49}, 
	Springer-Verlag(2004).
	
	\bibitem[U]{U} 
	\textsc{T.~Ueda}, 
	On the neighborhood of a compact complex curve with topologically trivial normal bundle, 
	J. Math. Kyoto Univ., \textbf{22} (1983), 583--607.
	
\end{thebibliography}
\end{document}